\documentclass{article}
\usepackage{hyperref}
\usepackage{amssymb,amsmath,amscd,amsthm}

\theoremstyle{plain}
\newtheorem{Theorem}{Theorem}[section]
\newtheorem{Lemma}[Theorem]{Lemma}
\newtheorem{Proposition}[Theorem]{Proposition}
\newtheorem{Corollary}[Theorem]{Corollary}

\theoremstyle{definition}
\newtheorem{Definition}[Theorem]{Definition}

\newcommand{\ssssarr}{\hbox to 15pt{\rightarrowfill}}
\newcommand{\sssarr}{\hbox to 20pt{\rightarrowfill}}
\newcommand{\ssarr}{\hbox to 30pt{\rightarrowfill}}
\newcommand{\sarr}{\hbox to 40pt{\rightarrowfill}}
\newcommand{\arr}{\hbox to 60pt{\rightarrowfill}}
\newcommand{\larr}{\hbox to 60pt{\leftarrowfill}}
\newcommand{\Arr}{\hbox to 80pt{\rightarrowfill}}
\newcommand{\mapdown}[1]{\Big\downarrow\rlap{$\vcenter{\hbox{$\scriptstyle#1$}}$}}

\newcommand{\matr}[1]{\begin{matrix}#1\end{matrix}}
\newcommand{\Aut}{\mathop{{\rm Aut}}\nolimits}

\newcommand{\be}{\begin{equation}}
\newcommand{\ee}{\end{equation}}
\newcommand{\beqn}{\begin{eqnarray}}
\newcommand{\eeqn}{\end{eqnarray}}
\newcommand{\nnb}{\begin{eqnarray*}}
\newcommand{\nne}{\end{eqnarray*}}

\newcommand{\dd}{{\tt d}} 
\newcommand{\lra}[2]{\langle #1, #2 \rangle}

\newcommand{\R}{\ensuremath{\mathbb{R}}}
\newcommand{\C}{\ensuremath{\mathbb{C}}}

\newcommand{\Z}{\ensuremath{\mathbb{Z}}}
\newcommand{\N}{\ensuremath{\mathbb{N}}}

\newcommand{\T}{\ensuremath{\mathbb{T}}}
\newcommand{\U}{\ensuremath{\mathrm{U}}}
\newcommand{\PU}{\ensuremath{\mathrm{PU}}}
\newcommand{\Sp}{\ensuremath{\mathrm{Sp}}}
\renewcommand{\Im}{\ensuremath{\mathrm{Im}}}
\renewcommand{\Re}{\ensuremath{\mathrm{Re}}}

\newcommand{\ham}{\ensuremath{\mathfrak{ham}}}
\newcommand{\fk}{\ensuremath{\mathfrak{k}}}
\newcommand{\fg}{\ensuremath{\mathfrak{g}}}
\newcommand{\g}{\ensuremath{\mathfrak{g}}}

\newcommand{\cH}{\ensuremath{\mathcal{H}}}
\newcommand{\cL}{\ensuremath{\mathcal{L}}}
\newcommand{\cM}{\ensuremath{\mathcal{M}}}

\newcommand{\bP}{\ensuremath{\mathbb P}}
\newcommand{\bS}{\ensuremath{\mathbb S}}

\newcommand{\bL}{\ensuremath{\mathbb L}}

\newcommand{\bxi}{\boldsymbol\xi}
\newcommand{\bc}{\boldsymbol{c}}

\newcommand{\ol}[1]{\ensuremath{\overline{#1}}}
\newcommand{\one}{\ensuremath{\mathbf{1}}}
\newcommand{\oline}{\overline} 
\newcommand{\subeq}{\subseteq} 
\newcommand{\la}{\langle} 
\newcommand{\ra}{\rangle} 
\renewcommand{\:}{\colon}

\begin{document}

\title{Momentum Maps for Smooth\\ Projective Unitary Representations} 
\author{Bas Janssens\footnote{
B.J.\ is supported by the NWO grant 
613.001.214 ``Generalised Lie algebra sheaves''.} {} and  
Karl-Hermann Neeb\footnote{K.-H.~N. is supported by  
the DFG-grant NE 413/7-2, Schwerpunktprogramm ``Darstellungstheorie''.}
}
\maketitle

\begin{abstract}
For a smooth projective unitary representation $(\overline{\rho},\cH)$
of a locally convex Lie group $G$, 
the projective space $\bP(\cH^{\infty})$ of smooth vectors is a locally convex
K\"ahler manifold. We show that the action of $G$ on $\bP(\cH^{\infty})$
is weakly Hamiltonian, and lifts 
to a Hamiltonian action of the central $\U(1)$-extension $G^{\sharp}$
obtained from the projective representation.
We identify the non-equivariance cocycles obtained from the 
weakly Hamiltonian action with those obtained from the projective representation, 
and give some integrality conditions on the image of the momentum map.
\end{abstract}

\section{Introduction}

Let $G$ be a locally convex Lie group with Lie algebra $\fg$, 
and let $\overline{\rho} \colon G \rightarrow \PU(\cH)$ be a projective unitary representation of $G$.
It is called \emph{smooth} if the set $\bP(\cH)^{\infty}$ of smooth rays
is dense in $\bP(\cH)$, a ray $[\psi] \in \bP(\cH)$ being called smooth if its orbit map
$G \rightarrow \bP(\cH) \colon g \mapsto \overline{\rho}(g)[\psi]$ is smooth.
For finite dimensional Lie groups, a projective representation is smooth if and only if it is continuous.
For infinite dimensional Lie groups, smoothness is a natural requirement.

In \cite[Theorem 4.3]{JN15}, we showed that for smooth projective unitary representations, 
the central extension 
\[G^{\sharp} := \{(g,U) \in G \times \U(\cH) \,;\, \overline{\rho}(g) = [U]\}\] 
of $G$ by $\U(1)$ is a central extension of locally convex Lie groups,
 in the sense that the projection $G^{\sharp} \rightarrow G$
is a homomorphism of Lie groups, as well as a principal $\U(1)$-bundle.
Moreover, the projective representation $\overline{\rho} \colon G \rightarrow \PU(\cH)$ of $G$
then lifts to a linear representation $\rho \colon G^{\sharp} \rightarrow \U(\cH)$ of $G^{\sharp}$,
with the property that $\rho(z) = z\one$ for all $z\in \U(1)$. 
If $\cH^{\infty} \subseteq \cH$ is the space of smooth vectors for $\rho$, then 
$\bP(\cH^{\infty})$ is equal to $\bP(\cH)^{\infty}$, the space of smooth rays for $\overline{\rho}$.

The main goal of these notes  is to reinterpret this central extension in the 
context of symplectic geometry 
of the projective space $\bP(\cH^{\infty})$ and its prequantum line bundle, the tautological 
bundle 
$\bL(\cH^{\infty}) \rightarrow \bP(\cH^{\infty})$.
In order to equip $\bP(\cH^{\infty})$ with a symplectic structure,
we need to consider it as a locally convex manifold.
For this, we need a locally convex topology on $\cH^{\infty}$ 
that is compatible with the $G^{\sharp}$-action. 

\begin{Definition}\label{defstrong}
The \emph{strong topology} on $\cH^{\infty}$ is the locally convex 
topology induced by the norm on $\cH$ and the seminorms 
\[p_{B} (\psi) := \sup_{\bxi \in B} \|\dd\rho_k(\bxi)\psi\|\,,\] 
where $B \subseteq (\fg^\sharp)^k$, $k \in \N$, 
runs over the bounded sets, and the derived representation 
$\dd\rho$ of $\fg^{\sharp}$ is extended to $(\fg^\sharp)^k$ by 
$\dd\rho_k(\xi_1, \ldots, \xi_k) := \dd\rho(\xi_1) \cdots \dd\rho(\xi_k)$. 
\end{Definition}


We will show that with this topology, $\bP(\cH^{\infty})$ becomes
a locally convex K\"ahler manifold with prequantum line bundle $\bL(\cH^{\infty}) \rightarrow \bP(\cH^{\infty})$. If we identify its tangent space 
$T_{[\psi]}\bP(\cH^{\infty})$ for any unit vector $\psi$ with 
$\{\delta v \in \cH^{\infty}\,;\, \langle \psi, \delta v \rangle = 0\}\,,$
then the symplectic form $\Omega$ on $\bP(\cH^{\infty})$
is given by 
\[
\Omega_{[\psi]}(\delta v,\delta w) = 2\Im(\delta v,\delta w)\,.
\] 
Similarly, the sphere $\bS(\cH^{\infty})$ becomes a locally convex principal
$\U(1)$-bundle over $\bP(\cH^{\infty})$, to which the
prequantum line bundle $\bL(\cH^{\infty})\rightarrow \bP(\cH^{\infty})$ is associated 
along the 
canonical representation $\U(1) \rightarrow \mathrm{GL}(\C)$.
The connection $\nabla$ on $\bL(\cH^{\infty})$ with curvature 
$R_{\nabla} = \Omega$ is associated to the connection 1-form
$\alpha$ on $\bS(\cH^{\infty})$, given by 
\[
\alpha_{\psi}(\delta v) = -i\lra{\psi}{\delta v}
\]
under the identification
$
T_{\psi}\bS(\cH^{\infty}) 
\simeq \{\delta v \in \cH^{\infty}\,;\, \mathrm{Re}\langle \psi, \delta v \rangle = 0\}\,.
$
%
%

The group $G$ acts on $\bP(\cH^{\infty})$ by K\" ahler automorphisms, hence in particular by symplectomorphisms.
The action of the central extension
$G^{\sharp}$ lifts to $\bL(\cH^{\infty})\rightarrow \bP(\cH^{\infty})$, on which it acts by holomorphic  
quantomorphisms (connection preserving bundle automorphisms).
%
If the projective representation of $G$ is faithful, then 
the central extension $G^{\sharp}$
is precisely the group of quantomorphisms of 
$(\bL(\cH^{\infty}),\nabla)$ that cover the $G$-action on $(\bP(\cH^{\infty}),\Omega)$.
\[  \matr{ 
G^\sharp & \circlearrowright  & \Aut(\bL(\cH^{\infty}),\nabla) \\ 
\mapdown{}&  & \mapdown{} \\ 
G & \circlearrowright & \Aut(\bP(\cH^{\infty}),\Omega)}
\]

We show that for any locally convex Lie group $G$, 
the action $G \times \bP(\cH^{\infty}) \rightarrow \bP(\cH^{\infty})$
obtained from a smooth projective unitary representation 
is \emph{separately smooth} in the following sense.

\begin{Definition}
An action $\alpha \colon G \times \mathcal{M} \rightarrow \mathcal{M}, 
(g,m) \mapsto \alpha_g(m)$ 
of a locally convex Lie group $G$ on a locally convex manifold
$\mathcal{M}$ is called \emph{separately smooth} if for every $g\in G$ and $m \in M$, the 
orbit map $\alpha_{m} \colon G \rightarrow \mathcal{M}, g \mapsto \alpha_g(m)$ 
and the action maps  $\alpha_{g} \colon \mathcal{M} \rightarrow \mathcal{M}$
are smooth.
\end{Definition}


For Banach Lie groups $G$, the action is a smooth map  
$G \times \bP(\cH^{\infty}) \rightarrow \bP(\cH^{\infty})$ by \cite[Theorem~4.4]{Ne10b}, but a certain lack of smoothness 
is unavoidable as soon as one goes to Fr\'echet--Lie groups.
Indeed, consider the unitary representation of the 
Fr\'echet--Lie group $G=\R^{\N}$  
on $\cH = \ell^2(\N,\C)$, defined by 
$\rho(\phi)\psi = (e^{i\phi_1}\psi_1, e^{i\phi_{2}}\psi_2, \ldots )$.
Then $\cH^{\infty} = \C^{(\N)}$ with the direct limit topology 
\cite[Example~3.11]{JN15}, but by \cite[Example 4.8]{Ne10b}, 
the action of $\fg$ on $\C^{\N}$ is discontinuous for \emph{any}
locally convex topology on $\cH^{\infty}$.
We therefore propose the following definition of (not necessarily smooth)
Hamiltonian actions on locally convex manifolds.

\begin{Definition}
An action $\alpha \colon G \times \mathcal{M} \rightarrow \mathcal{M}$ 
of a locally convex Lie group $G$ on a locally convex, symplectic manifold
$(\mathcal{M},\Omega)$ is called:
\begin{itemize}
\item \emph{Symplectic} if it is separately smooth, and $\alpha_{g}^*\Omega = \Omega$ for all $g\in G$.
\item \emph{Weakly Hamiltonian} if it is symplectic, and $i_{X_{\xi}}\Omega$ is 
exact for all $\xi \in \fg$, where $X_{\xi}$ is the fundamental vector field of $\xi$
on $\bP(\cH^{\infty})$.
\item \emph{Hamiltonian} if, moreover, $i_{X_{\xi}}\Omega = d\mu(\xi)$ for a $G$-equivariant momentum map $\mu \colon \mathcal{M} \rightarrow \fg'$ into the continuous dual of $\fg$, which is smooth if $\fg'$ is equipped 
with the topology of uniform convergence on bounded subsets.
\end{itemize}
\end{Definition}
%
Our main result is that the action of $G^{\sharp}$ on $\bP(\cH^{\infty})$ is Hamiltonian in the sense of the above definition.
\begin{Theorem}\label{Thmhamaction}
The action of $G^{\sharp}$ on $(\bP(\cH^{\infty}),\Omega)$ is Hamiltonian,
with momentum map
$\mu \colon \bP(\cH^{\infty}) \rightarrow \fg^{\sharp}{}'$ given by 
\be
\mu_{[\psi]}(\xi) = \frac{\lra{\psi}{i\dd\rho(\xi) \psi}}{\lra{\psi}{\psi}}\,.
\ee
\end{Theorem}
Since the $G$-action on $\bP(\cH^{\infty})$ factors through the action of $G^{\sharp}$, we immediately obtain
the following corollary.
\begin{Corollary}
The action of $G$ on $(\bP(\cH^{\infty}),\Omega)$ is weakly Hamiltonian.
\end{Corollary}

Note that the (classical) momentum associated to $\xi \in \fg^{\sharp}$ at 
$[\psi] \in \bP(\cH^{\infty})$ is precisely the corresponding (quantum mechanical) expectation of the 
self-adjoint operator (observable) $i\dd\rho(\xi)$ in the state $[\psi]$. 

%

Sections \ref{Section1} and \ref{Section2} of this paper are concerned with the proof of
Theorem~\ref{Thmhamaction}.
In Section~\ref{Section1}, we show in detail that $\bP(\cH^{\infty})$ is a locally convex, prequantisable K\"ahler manifold, and 
in Section~\ref{Section2}, we use this to show that
%
the action of $G^{\sharp}$ on $\bP(\cH^{\infty})$ is Hamiltonian, and lifts
to the prequantum line bundle $\bL(\cH^{\infty}) \rightarrow \bP(\cH^{\infty})$.  

In the second half of this paper, we give some applications of this symplectic picture 
to projective representations. 
In Section~\ref{Section3}, we calculate the Kostant--Souriau cocycles 
associated to the Hamiltonian action, and show that these are precisely the 
Lie algebra cocycles that one canonically obtains from a 
projective unitary representation and a smooth ray, cf.\  \cite{JN15}.
We then prove an integrality result
for characters of the stabiliser group that 
one obtains as the image of the momentum map. 
Finally, in Section~\ref{sec:diffeosmoothness},
we close with some remarks on smoothness of the action in the context of diffeological spaces.

Momentum maps have been 
introduced into representation theory by Normal Wildberger \cite{Wi92}.  
Studying the image of the momentum map has proven to be 
an extremely powerful tool in the analysis of 
unitary representations, in particular 
to obtain information on upper and lower bounds of 
spectra  (\cite{AL92}, \cite{Ne95, Ne00, Ne10a}). 
Smoothness properties of the linear Hamiltonian action on the space 
$\cH^\infty$ of smooth vectors of a unitary representation and the corresponding 
momentum map $\mu_\psi(\xi) := \la \psi, \dd\pi(\xi) \psi \ra$
have been studied by P.~Michor in the context of convenient calculus in~\cite{Mi90}. 

\section{The locally convex symplectic space $\bP(V)$}\label{Section1}
In order to equip $\bP(\cH^{\infty})$ with a symplectic structure,
we need to consider it as a locally convex manifold (in the sense of \cite[Def.\ 9.1]{Mi80}).
Later on, it will be important to choose a locally convex topology on $\cH^{\infty}$
that is compatible with the group action, but for now, it suffices if the scalar product 
on $\cH^{\infty} \subseteq \cH$ is continuous.
 We will go through the standard constructions of projective 
 geometry, using only a complex, locally convex space $V$
with continuous hermitian scalar product 
 $\lra{\,\cdot\,}{\,\cdot\,}\colon V \times V \rightarrow \C$ 
(antilinear in the first and linear in the second argument).

\begin{Proposition} 
The projective space $\bP(V)$ is a complex manifold modelled on 
locally convex spaces. The tautological line bundle 
$\mathbb{L}(V) \rightarrow \bP(V)$ is a locally convex, holomorphic 
line bundle over $\bP(V)$.
\end{Proposition}
\begin{proof}
We equip $\bP(V)$ with the Hausdorff topology 
induced by the quotient 
map $V - \{0\} \rightarrow \bP(V)$. 
The open neighbourhood
\be
U_{[\psi]} := \{[\chi] \in \bP(V)\,;\, \langle \psi, \chi \rangle \neq 0\}
\ee
is then charted by the hyperplane 
\begin{equation}
T_{[\psi]} := \{v \in V\,;\, \langle \psi, v \rangle = 0\}\,,
\end{equation}
and the chart $\kappa_{\psi} \colon U_{[\psi]} \rightarrow T_{[\psi]}$, 
defined by $\kappa_{\psi}([\chi]) := \langle\psi, \chi \rangle^{-1}\chi - \psi$ 
(cf.~\cite[\S V.1]{Ne01}).
Note that the map $\kappa_{\psi}$ depends on the choice of representative
$\psi \in [\psi]$, which we will assume to be of unit length. 
The inverse chart is $\kappa_{\psi}^{-1}(v) = [\psi + v]$.
We have $\kappa_{z\psi}([\chi]) = z \kappa_{\psi}([\chi])$ for $z \in \U(1)$.
More generally, the transition function 
$\kappa_{\psi\psi'} \colon \kappa_{\psi}(U_{[\psi]} \cap U_{[\psi']}) \rightarrow \kappa_{\psi'}(U_{[\psi]} \cap U_{[\psi']})$ 
is given by
$v \mapsto \frac{\psi + v}{\langle\psi', \psi + v\rangle} - \psi'$.
Since $v \mapsto \langle\psi', \psi + v\rangle$ is continuous and nonzero on 
$\kappa_{\psi}(U_{[\psi]} \cap U_{[\psi']})$, the transition functions are 
holomorphic, making 
$\bP(V)$ into a complex manifold.

For $\mathbb{L}(V) = V - \{0\}$, define the charts 
$\Lambda_{\psi} \colon T_{[\psi]} \times \C \rightarrow \mathbb{L}(V)$
by $\Lambda_{\psi} (v,z) = z(\psi + v)$.
The transition functions 
\[\Lambda_{\psi\psi'} \colon \kappa_{\psi}(U_{[\psi]} \cap U_{[\psi']})\times \C
\rightarrow \kappa_{\psi'}(U_{[\psi]} \cap U_{[\psi']})\times \C\] 
are given by
$(v,z) \mapsto (\kappa_{\psi\psi'}(v),\lra{\psi'}{\psi + v}z)$. Since these are 
holomorphic isomorphisms of trivial locally convex line bundles, 
the result follows. 
\end{proof}

For $\|\psi\| = 1$, we 
identify the tangent vectors $\delta v, \delta w \in T_{[\psi]} \bP(V)$ at the point 
$[\psi] \in \bP(V)$ with their coordinates $\delta v, \delta w \in T_{[\psi]}$ 
by the tangent map $T_{[\psi]}(\kappa_\psi)$. 
Accordingly, we define the Hermitean form $H$ on $\bP(V)$ by
\be \label{Hermitean}
H_{[\psi]}(\delta v, \delta w) := 2\langle \delta v, \delta w \rangle\,.
\ee
Note that this does not depend on the choice of chart $\kappa_{\psi}$.

\begin{Proposition}
Equipped with the Hermitean forms $H_{[\psi]}$ 
of equation \eqref{Hermitean}, $\bP(V)$ is a Hermitean manifold.
\end{Proposition}

\begin{proof} As compatibility with the complex structure 
$J(\delta v) = i \delta v$ is clear, the only thing to show is that $H$
is smooth.
Using that the transition map 
\[ D_v\kappa_{\psi\psi'} \: T_{[\psi]} \to T_{[\psi']}, \quad \mbox{ for } \quad 
\psi' = \frac{\psi + v}{\|\psi + v\|} \] 
is given by 
\[D_{v}\kappa_{\psi, \psi'}(\delta v) = 
\textstyle
\frac{1}{\|\psi + v\|}\left(\delta v - \lra{\frac{\psi + v}{\|\psi + v\|}}{\delta v}
\frac{\psi + v}{\|\psi + v\|}\right)\,,\]
one sees that in local coordinates for $T^2\bP(V)$, 
the map $T_{[\psi]} \times T_{[\psi]} \times T_{[\psi]} \rightarrow \C$ is 
\be\label{localomega}
H_{v}(\delta v,\delta w) = 2 \left(\frac{1}{1 + \|v\|^2} \lra{\delta v}{\delta w} - 
\frac{1}{(1 + \|v\|^2)^2}\lra{\delta v}{v}
\lra{v}{\delta w} \right)\,,
\ee
which is evidently smooth. 
\end{proof}

As the real and imaginary parts of $H$, we obtain the Fubini--Study metric 
\[G_{[\psi]}(\delta v , \delta w) = 2\mathrm{Re}\lra{\delta v}{\delta w}\]
and the 2-form 
\be\label{sympform}
\Omega_{[\psi]}(\delta v,\delta w) = 2\mathrm{Im}\langle \delta v, \delta w \rangle\,.
\ee

The 2-form $\Omega$ is nondegenerate in the `weak' sense that 
$\Omega(\delta v, \delta w) = 0$ for all $\delta w$ implies $\delta v = 0$.
In order to show that $\Omega$ is a symplectic form, and hence that $\bP(V)$ is K\"ahler, 
it thus suffices to prove that it is closed. We will do this by showing that $\Omega$
is the curvature of a prequantum bundle.

\begin{Proposition}
The sphere $\bS(V) = \{\psi \in V \,;\, \|\psi\| = 1 \}$ is a locally convex manifold, and the projection 
$\bS(V) \rightarrow \bP(V)$ is a principal $\U(1)$-bundle.
\end{Proposition}
\begin{proof}
The sphere inherits the Hausdorff topology from its inclusion in $V$. 
The locally convex space 
\be\label{schart}
T_{\psi} := \{v \in V\,;\, \mathrm{Re}\langle \psi, v \rangle = 0\} \subseteq V
\ee
can be naturally identified with the open neighbourhood 
\be
U_{\psi} := \{\chi \in \bS(V) \,;\, \mathrm{Re}\langle \psi, \chi\rangle > 0\}
\ee
of $\psi$ 
by the chart $\kappa \colon U_{\psi} \rightarrow T_{\psi}$ with 
$\kappa_{\psi}(\chi) = (\mathrm{Re} \langle \psi,\chi \rangle)^{-1}\chi - \psi$,
which has inverse $\kappa_{\psi}^{-1}(v) = \frac{\psi + v}{\|\psi + v\|}$.
If $\psi$ and $\psi'$ are not antipodal, then the 
transition $U_{\psi\psi'}:= U_{\psi} \cap U_{\psi'}$ is nonempty, and 
the transition function 
$\kappa_{\psi}(U_{\psi \psi'}) \rightarrow \kappa_{\psi'}(U_{\psi \psi'})$
is given by $v \mapsto \frac{\psi + v}{\mathrm{Re} \langle\psi', \psi + v\rangle} - \psi'$.
This is continuous for the strong topology that $T_{\psi}$ and $T_{\psi'}$ inherit from
$V$ because the scalar product $\langle\,\cdot\,, \,\cdot \,\rangle$ is 
continuous, and $\mathrm{Re} \langle\psi', \psi + v\rangle$ is nonzero on 
$\kappa_{\psi}(U_{\psi} \cap U_{\psi'})$.
In particular, the tangent space
$T_{\psi} \bS(V)$ can be canonically identified with $T_{\psi}$.

The canonical projection $\bS(V) \rightarrow \bP(V)$ is a smooth principal
$\U(1)$-bundle, with local trivialisation $\tau_{\psi} \colon T_{[\psi]} \times \U(1) \rightarrow \bS(V)$
given by $\tau_{\psi}(v, z) := z \frac{\psi + v}{\|\psi + v\|}$. (Note that this depends on the representative $\psi$
of $[\psi]$.)
For $[\chi]$ in its image $U_{\psi} \cup U_{-\psi} \cup U_{i\psi} \cup U_{-i\psi}$,
we have $z_{\psi}([\chi]) = \frac{\langle\psi, \chi\rangle}{|\langle\psi, \chi\rangle|}$
and $z_{\psi'}([\chi]) = \frac{\langle\psi', \chi\rangle}{|\langle\psi', \chi\rangle|}$,
so the clutching 
functions $g_{\psi\psi'} \colon U_{[\psi]} \cap U_{[\psi']}\rightarrow \T$ are 
\[g_{\psi\psi'}([\chi]) = \frac{\langle\psi', \chi\rangle}{\langle\psi, \chi\rangle}\Big/ 
\left|\frac{\langle\psi', \chi\rangle}{\langle\psi, \chi\rangle}\right|\,.
\qedhere\]
\end{proof}

Identifying $T_{\psi}\bS(V)$ with $T_{\psi}$ in \eqref{schart}, 
we define the 1-form $\alpha$ on $\bS(V)$ by
\be\label{sconn}
\alpha_{\psi}(\delta v) = -i\lra{\psi}{\delta v}\,.
\ee
\begin{Proposition}
The form $\alpha$ is a connection 1-form on $\bS(V) \rightarrow \bP(V)$
with curvature $\Omega$.
\end{Proposition}
\begin{proof}
We start by showing that $\alpha$ is smooth.
Using the derivative 
\[D_{v}\kappa_{\psi\psi'}(\delta v) 
= \frac{1}{\|\psi + v\|}
\Big(\delta_v - \frac{\Re \la v, \delta v \ra}{1 + \|v\|^2}(\psi + v)\Big)
= \frac{\delta v}{\|\psi + v\|}
-
\frac{\mathrm{Re}\lra{v}{\delta v}}{\|\psi + v\|^3}(\psi + v)
\]
for the transition function with $\psi' = \frac{\psi + v}{\|\psi + v\|}$, 
one sees that $\alpha$ is represented by the function $T_{\psi} \times T_{\psi} \rightarrow \R$
given by $(v,\delta v) \mapsto \frac{1}{1 + \|v\|^2}\mathrm{Im}\lra{v}{\delta v}$, which is evidently smooth.

If we identify $T_{\psi}\bS(V) $ with $ T_{\psi}$ and $T_{z\psi}\bS(V) $ with $ T_{z\psi}$,
then the pushforward $R_{z*} \colon T_{\psi} \rightarrow T_{z\psi}$ of the 
$\U(1)$-action is
$R_{z*}(\delta v) = z \delta v$.  
It follows that $\alpha$ is 
$\U(1)$-invariant, 
\[(R_{z}^*\alpha_{\psi})(\delta v) = \alpha_{z\psi}(z\delta v) = -i\langle z\psi, z\delta v\rangle =
 -i\langle \psi, \delta v\rangle = \alpha_{\psi}(\delta v),\]
and since the vector field $X_1$ generated by the $\U(1)$-action on $\bS(V)$ is 
$X_{1}(\psi) = \frac{d}{dt}\big|_{t =0} e^{it}\psi = i\psi$, we have 
$\alpha_{\psi} (X_{1}(\psi)) = 1$, so that $\alpha$ is a principal connection 1-form on 
$\bS(V) \rightarrow \bP(V)$.
If we introduce the constant vector fields $\delta v$ and $\delta w$ on $U_{\psi} \subseteq \bS(V)$,
then at $v=0$, we have
\be
d\alpha_{v}(\delta v, \delta w) = L_{\delta v}\alpha_{v}(\delta w) - L_{\delta v}\alpha_{v}(\delta w)
= 2\mathrm{Im}\lra{\delta v}{\delta w}\,,
\ee
which agrees with the local expression 
\eqref{localomega} for $\Omega_{[\psi]}(\delta v,\delta w)$ at $v=0$, as required. 
\end{proof}

In particular, $\Omega$ is closed, so $\bP(V)$ is a K\"ahler manifold.
Since the tautological line bundle is associated to 
$\bS(V)$ in the sense that 
$\mathbb{L}(V) := \bS(V) \times_{\U(1)}\C$, we have the following result 
(see also \cite{Ne01}). 

\begin{Theorem}\label{Thm:projspace}
The projective space $\bP(V)$ with Hermitean form $H$ 
is a locally convex K\"ahler manifold. 
The tautological bundle $\mathbb{L}(V) \rightarrow \bP(V)$, equipped with the connection 
inherited from the $\U(1)$-principal 1-form $\alpha$, 
is a prequantum line bundle for the corresponding symplectic 
form $\Omega$. 
\end{Theorem}


\section{Hamiltonian action of $G^{\sharp}$ on $\bP(\cH^{\infty})$}\label{Section2}
%


We return to the situation of a smooth, projective, unitary representation 
$\overline{\rho}$ of $G$, and the corresponding unitary representation $\rho$ of $G^{\sharp}$.
In order to obtain a Hamiltonian action of $G^{\sharp}$
on $\bP(\cH^{\infty}),$
we need a locally convex topology on $\cH^{\infty}$ 
that is compatible with with the $G^{\sharp}$-action. 
We will equip $\cH^{\infty}$ with the \emph{strong topology} of Definition~\ref{defstrong}.
As the scalar product $\lra{\,\cdot\,}{\,\cdot\,} \colon \cH^{\infty} \times \cH^{\infty} \rightarrow \C$
is manifestly continuous, Theorem~\ref{Thm:projspace} applies to $\bP(\cH^{\infty})$.

%

\begin{Proposition}\label{prop:sepsmooth}
The group action $G^{\sharp} \times \cH^{\infty} \rightarrow \cH^{\infty}$
is separately smooth for the strong topology.
\end{Proposition}
\begin{proof}
For fixed $g\in G^{\sharp}$, we show that the linear map 
$\rho(g) \colon \cH^{\infty} \rightarrow \cH^{\infty}$
is continuous. If $B \subseteq (\fg^\sharp)^k$ is bounded, then so is 
$\mathrm{Ad}_{g}(B)$, as the action 
$\mathrm{Ad}_{g} \colon (\fg^\sharp)^k \rightarrow (\fg^\sharp)^k$ 
of $g$ in the $k$-fold product of the adjoint representation is a homeomorphism.
From
\begin{eqnarray*}
p_{B}(\rho(g)\psi) &=& \sup_{\bxi \in B} \|\dd\rho_k(\bxi)\rho(g)\psi\|
 =  \sup_{\bxi \in B} \|\rho(g)\dd\rho_k(\mathrm{Ad}_{g^{-1}}(\bxi))\psi\|\\
 &= & p_{\mathrm{Ad}_{g^{-1}}(B)}(\psi)\,,
\end{eqnarray*}
we then see that $\rho(g)$ is strongly continuous.
If we fix $\psi \in \cH^{\infty}$, then the orbit map $g \mapsto \rho(g)\psi$ is smooth 
in the \emph{norm} topology on $\cH^{\infty} \subseteq \cH$ by definition, but we still need
to show that it is smooth in the \emph{strong} topology. This follows from 
\cite[Lemma 3.24]{JN15}.
\end{proof}

Our (somewhat laborious) proof of Theorem \ref{Thm:projspace} 
now allows us to apply Proposition~\ref{prop:sepsmooth} 
in local coordinates, yielding the following result.

\begin{Proposition}
The locally convex Lie group $G$ acts separately smoothly 
on $\bP(\cH^{\infty})$ by K\"ahler automorphisms.
This action is covered by a separately smooth action of 
$G^{\sharp}$ on the prequantum line bundle 
$\mathbb{L}(\cH^{\infty}) \rightarrow \bP(\cH^{\infty})$ by 
holomorphic, connection-preserving bundle automorphisms. 
\end{Proposition}
\begin{proof}
In local coordinates, the action of $G^{\sharp}$ looks like
$T_{[\psi]} \rightarrow T_{\overline{\rho}(g)[\psi]} \colon v \mapsto \rho(g)v$
on $\bP(\cH^{\infty})$, like
$T_{\psi} \rightarrow T_{\rho(g)\psi} \colon v \mapsto \rho(g)v$
on $\bS(\cH^{\infty})$, and like
$T_{[\psi]} \times \C \rightarrow T_{\overline{\rho}(g)[\psi]} \times \C \colon v\oplus z 
\mapsto \rho(g)v \oplus z$
on $\bL(\cH^{\infty})$.
%
It thus follows from Proposition~\ref{prop:sepsmooth} that the group action is 
separately smooth, and a holomorphic line bundle isomorphism of
$\bL(\cH^{\infty}) \rightarrow \bP(\cH^{\infty})$.
In local coordinates, the pushforwards 
$\rho(g)_{*} \colon T_{[\psi]} \rightarrow T_{\overline{\rho}(g)[\psi]}$
and 
$\rho(g)_{*} \colon T_{\psi} \rightarrow T_{\rho(g)\psi}$
are simply given by $\delta v \mapsto \rho(g)\delta v$, 
so $\rho(g)^*H = H$ and $\rho(g)^*\alpha = \alpha$ follow from
unitarity of $\rho(g)$ and
the definitions \eqref{Hermitean} and \eqref{sconn}.
\end{proof}

For $\xi \in \fg^{\sharp}$, the fundamental vector field $X_{\xi}(\psi) = \dd\rho(\xi)\psi$ on 
$\bS(\cH^{\infty})$ is smooth, as it is
given in local coordinates $v \in T_{\psi}$ by 
\[X_{\xi}(u) = \dd\rho(\xi) (\psi + v)   - \mathrm{Re}\langle \psi, \dd\rho(\xi) v\rangle(\psi + v)\,.\]
Since $L_{X_{\xi}} \alpha = 0$, we have $d(i_{X_{\xi}} \alpha) + i_{X_{\xi}} d\alpha = 0$,
so since $\Omega = d\alpha$, we find
\begin{equation}\label{exactly}
i_{X_{\xi}} \Omega = d(-i_{X_{\xi}}\alpha)\,.
\end{equation}
We therefore find the 
 comomentum map $\fg^{\sharp} \rightarrow C^{\infty}(\bP(\cH^{\infty})), 
\xi \mapsto \mu(\xi)$ with
\begin{equation}\label{rawmoment}
\mu_{[\psi]}(\xi) = \alpha_{\psi}(-X_{\xi}(\psi))\,.
\end{equation}
This evaluates to $\langle \psi, i\dd\rho(\xi) \psi\rangle$,
the expectation in the state $[\psi]$ of the essentially selfadjoint operator
$i\dd\rho(\xi)$ (cf.\ Definition~\ref{defstrong}), 
which is the observable corresponding to the 
symmetry generator $\xi \in \fg$.
Note that for fixed $\xi$, the expression $\psi \mapsto \alpha_{\psi}(-X_{\xi}(\psi))$ is independent 
of the unit vector $\psi \in [\psi]$, and smooth 
because both $\alpha$ and $X_{\xi}$ are smooth. 

We now prove the theorem announced in the introduction 
(Theorem~\ref{Thmhamaction}): 
\begin{Theorem} 
The action of $G^{\sharp}$ on $(\bP(\cH^{\infty}),\Omega)$ is Hamiltonian,
with momentum map
$\mu \colon \bP(\cH^{\infty}) \rightarrow (\fg^{\sharp})'$ given by 
\be\label{eq:momentum}
\mu_{[\psi]}(\xi) = \frac{\lra{\psi}{i\dd\rho(\xi) \psi}}{\lra{\psi}{\psi}}\,.
\ee
This is a smooth, $G^{\sharp}$-equivariant map into the continuous dual $(\fg^{\sharp})'$, equipped
with the topology of uniform convergence on bounded subsets. 
\end{Theorem}
\begin{proof}
Since the $G^{\sharp}$-action preserves $\alpha$, it preserves $\Omega$, and $i_{X_{\xi}}\Omega$
is exact by equation~(\ref{exactly}). 
Combining \eqref{exactly} and \eqref{rawmoment}, we have $i_{X_{\xi}}\Omega = d\mu_{\xi}(\xi)$.
The momentum map 
is equivariant by
\[\mu_{[g\psi]}(\xi) = \langle \rho(g)\psi, i\dd\rho(\xi) \rho(g)\psi \rangle
= \langle \psi , i\dd\rho(\mathrm{Ad}_{g^{-1}} (\xi)) \psi\rangle \,.\]
To prove that $\mu$ is smooth, consider its pullback to $\bS(\cH^\infty)$, which is 
the restriction to $\bS(\cH^{\infty})$ of the map 
$\widehat{\mu}\colon \cH^{\infty} \rightarrow (\fg^{\sharp})'$ defined by 
$\widehat{\mu}_{\psi} = \lra{\psi}{\dd\rho(\,\cdot\,)\psi}$.
Note that the map 
\[
\cH^{\infty} \rightarrow \mathrm{Lin}(\fg^{\sharp},\cH^{\infty}), 
\quad \psi \mapsto \dd\rho(\,\cdot\,)\psi\]
is linear, and continuous if $\mathrm{Lin}(\fg^{\sharp},\cH^{\infty})$ is equipped with the topology of uniform convergence on bounded subsets of $\fg$. 
As the scalar product $\cH^{\infty} \times \cH^{\infty} \rightarrow \R$ is continuous, 
the linear map $\cH^{\infty} \times \cH^{\infty} \rightarrow (\fg^{\sharp})', 
 (\psi,\chi) \mapsto \lra{\psi}{\dd\rho(\,\cdot\,)\chi}$ is also continuous, and hence smooth.   
Since $\widehat{\mu}$ is the composition of this map with the (smooth) diagonal map
$\cH^{\infty} \rightarrow \cH^{\infty} \times \cH^{\infty}$, the result follows.
%
%
%
%
%
%
%
\end{proof}

\section{Cocycles for Hamiltonian actions}\label{Section3}


A symplectic action of a locally convex Lie group $G$ 
on a locally convex, symplectic manifold 
$(\cM,\Omega)$ gives rise to \emph{Kostant--Souriau cocycles}. 
\begin{Proposition}[Kostant--Souriau cocycles]\label{KScocycles}
For every $m\in \cM$, the map $\omega_{m} \colon \fg \times \fg \rightarrow \R$
defined by $\omega_{m}(\xi,\eta) = \Omega_{m}(X_{\xi},X_{\eta})$ is a continuous 2-cocycle.
If $\cM$ is a K\"ahler manifold, then 
$\omega_{m} = \mathrm{Im} \,h_{m}$ for a
continuous, positive semidefinite, Hermitean form  
$h_{m} \colon \fg_{\C} \times \fg_{\C} \rightarrow \C$. 
\end{Proposition}
\begin{proof}
Since the action is symplectic, $L_{X_{\xi}} \Omega = 0$, and we have
\[L_{X_{\xi}}\Omega(X_{\eta},X_{\zeta}) = \Omega([X_{\xi},X_{\eta}],X_{\zeta})
+ \Omega(X_{\eta}, [X_{\xi},X_{\zeta}])\,.\]
As $\Omega$ is closed, it follows that for all $\xi,\eta,\zeta \in \fg$,
\nnb
0 = d\Omega(X_{\xi},X_{\eta},X_{\zeta}) &=&
(L_{X_{\xi}}\Omega(X_{\eta},X_{\zeta}) + \mathrm{cycl.})
- (\Omega([X_{\xi},X_{\eta}],X_{\zeta})+ \mathrm{cycl.})\\
&=& \Omega(X_\eta, [X_\xi,X_{\zeta}]) + \mathrm{cycl}. = \delta \omega(\xi,\eta,\zeta)\,,
\nne
and $\omega_{m}$ is a cocycle for every $m\in \cM$.
Since the orbit map $g \mapsto \alpha_{g}(m)$ is smooth, the map 
$\xi \mapsto X_{\xi}(m)$ is continuous. Since 
$\Omega_{m} \colon T_{m}\cM \times T_{m}\cM \rightarrow \R$ is smooth,  
the cocycle $\omega_{m}$ is continuous. 
If $\cM$ is K\"ahler, then 
$\Omega_{m}$ is the imaginary part of a positive definite Hermitean form 
$H_{m}$ on $T_{m}\cM$.
We then have $\omega_{m} = \mathrm{Im} \, h_{m}$ for 
the pullback
$h_{m} \colon \fg_{\C} \times \fg_{\C} \rightarrow \C$ 
of $H_{m}$ along the complexification 
$D_{m}\alpha \colon \fg_{\C} \rightarrow T_{m}\cM$
of the derivative of the orbit map.
\end{proof}

\subsection{Cocycles for projective unitary representations}

For the weakly Hamiltonian action of $G$ on $\bP(\cH^{\infty})$ derived from a smooth projective
unitary representation,
the Kostant--Souriau cocycles are given by 
\be \label{eq:stompje}
\omega_{[\psi]}(\xi,\eta) = 2\mathrm{Im}\lra{\dd\rho(\xi^{\sharp}) \psi}{\dd\rho(\eta^{\sharp})\psi}\,,
\ee
where $\xi^{\sharp},\eta^{\sharp}\in \fg^{\sharp}$ are arbitrary lifts of 
$\xi,\eta \in \fg$. (This does not depend on the choice of lift because 
$\lra{\psi}{i\dd\rho(\xi^{\sharp})\psi}$ is real.)

In particular, we see that the Kostant--Souriau cocycles associated to a smooth 
projective representation arise as the image of the momentum map 
$\mu \colon \bP(\cH^{\infty}) \rightarrow (\fg^{\sharp})'$, concatenated with the 
differential 
$\delta \colon \fg^{\sharp} \rightarrow \mathcal{Z}^2(\fg)$ that maps $\lambda \in 
(\fg^{\sharp})'$
to the 2-cocycle 
\[ (\delta\lambda) (\xi,\eta) := \lambda([\xi^{\sharp},\eta^{\sharp}]),\] 
which is again independent of the choice of lift.

\begin{Proposition}
For the weakly Hamiltonian action of $G$ on $\bP(\cH^{\infty})$ derived from a smooth projective unitary representation, we have
$\omega_{[\psi]} = \delta \mu_{[\psi]}$.
\end{Proposition}
\begin{proof}
This is a direct computation. From \eqref{eq:momentum}, we obtain 
\nnb
\delta\mu_{[\psi]}(\xi,\eta) & = & \frac{\lra{\psi}{i\dd\rho([\xi^{\sharp},\eta^{\sharp}])\psi}}{\lra{\psi}{\psi}}\\
&=& -i\left(\frac{\lra{\dd\rho(\xi^{\sharp})\psi}{\dd\rho(\eta^{\sharp})\psi}}{\lra{\psi}{\psi}} - 
\frac{\lra{\dd\rho(\eta^{\sharp})\psi}{\dd\rho(\xi^{\sharp})\psi}}{\lra{\psi}{\psi}}\right)\\
&=& 2\,\frac{\mathrm{Im}\lra{\dd\rho(\xi^{\sharp}) \psi}{\dd\rho(\eta^{\sharp})\psi}}{\lra{\psi}{\psi}}\,,
\nne
which equals $\omega_{[\psi]}(\xi,\eta)$ for $\|\psi\| = 1$ by \eqref{eq:stompje}.
\end{proof}

From a smooth projective unitary representation, we thus get not only a \emph{class} 
$[\omega_{[\psi]}] \in H^2(\fg,\R)$ in continuous Lie algebra cohomology, 
but a distinguished set
$\mathcal{C} := \{\omega_{[\psi]}\,;\, [\psi]\in \bP(\cH^{\infty})\} \subseteq 
\mathcal{Z}^2(\fg)$
of (cohomologous, cf.\ \cite{JN15}) cocycles. 
As both $\mu$ and $\delta$ are $G$-equivariant, this set 
$\mathcal{C} = \mathrm{Im}(\delta \circ \mu)$ of cocycles is $G$-invariant, 
and every $\omega \in \mathcal{C}$ is the imaginary part of a continuous, positive
semidefinite, Hermitean form on $\fg_{\C}$ by Proposition \ref{KScocycles}.
This sheds geometric light on Propositions~6.6, 6.7 and 6.8 of \cite{JN15}.

\subsection{Characters of the stabiliser group}\label{charstab}

The derivative of the momentum map $\mu \colon \bP(\cH^{\infty}) \rightarrow 
(\fg^{\sharp})'$ is given (for $\|\psi\|=1$) by
\be
D_{[\psi]}\mu(\delta v)(\xi) = 2\mathrm{Re}\langle i\dd\rho(\xi)\psi, \delta v\rangle\,.
\ee
This has some interesting consequences.
We denote the \emph{real} inner product on $\cH^{\infty}$ by 
$(v,w)_{\R} := 2\mathrm{Re} \lra{v}{w}$, and the orthogonal complement 
with respect to $(\,\cdot\,,\,\cdot\,)_{\R}$ by $\perp_{\R}$.
Then the kernel $\mathrm{Ker}(D_{[\psi]}\mu) \subseteq T_{[\psi]} = (\C\psi)^{{\perp}_{\R}}$ 
is precisely the \emph{real} orthogonal complement
$(i\R\psi \oplus i\dd\rho(\fg^{\sharp})\psi)^{\perp_{\R}}$ in $\cH^{\infty}$.

\begin{Proposition}
The derivative $D_{[\psi]}\mu \colon T_{[\psi]} \rightarrow \fg'$ is injective if and only if 
$\dd\rho(\fg^{\sharp})\psi$ spans $\psi^{\perp_{\R}} \subset \cH$ as a real Hilbert space, and zero if and only if 
the identity component $G_{0}$ stabilises $[\psi]$.
\end{Proposition}
\begin{proof}
The first statement follows immediately from the formula for the kernel.
For the second statement, note that $D_{[\psi]}\mu = 0$ is equivalent to
$\dd\rho(\fg^{\sharp})\psi = i\R\psi$.
By the Fundamental Theorem of Calculus for locally convex spaces, 
$G_{0}$ stabilises $[\psi] \in \bP(\cH^{\infty})$ if and only if 
$\fg$ stabilises $[\psi]$, which is the case if and only if 
$\dd\rho(\fg^{\sharp})\psi \subseteq i\R \psi$.
\end{proof}

We denote the stabiliser of $\lambda \in (\fg^{\sharp})'$ 
under the coadjoint representation by
$G^{\sharp}_{\lambda}$. Further, we denote by
\[G^{\sharp}_{[\psi]} := \{g\in G^{\sharp}\,: \, [\rho(g)\psi] = [\psi]\}\]
 the preimage in $G^{\sharp}$
of the stabiliser $G_{[\psi]}$ of $[\psi] \in \bP(\cH^{\infty})$, and we denote  
\[\fg^{\sharp}_{[\psi]} := \{\xi \in \fg^{\sharp}\,;\, \dd\rho(\xi) \psi \in i\R\psi\}\,.\]

\begin{Proposition}
For every $[\psi] \in \bP(\cH^{\infty})$, we have 
$G^{\sharp}_{[\psi]} \subseteq G^{\sharp}_{\mu_{[\psi]}}$. 
\end{Proposition}
\begin{proof}
Since the momentum map is $G^{\sharp}$-equivariant, we have $g \in G_{\mu_{[\psi]}}$
if and only if 
\be\label{eq:statemu}
\frac{\lra{\rho(g)\psi}{i\dd\rho(\xi)\rho(g)\psi}}{\lra{\rho(g)\psi}{\rho(g)\psi}}
=
\frac{\lra{\psi}{i\dd\rho(\xi)\psi}}{\lra{\psi}{\psi}} 
\ee
for all $\xi \in \fg$. This is clearly the case if $g\in G^{\sharp}_{[\psi]}$.
%
\end{proof}
\begin{Proposition} \label{IntegralCharacters}
The restriction of $-i\mu_{[\psi]} \colon \fg^{\sharp} \rightarrow i\R$ to $\fg^{\sharp}_{[\psi]}$
is a Lie algebra character.
It integrates to a group character on any Lie subgroup of~$G^{\sharp}_{[\psi]}$.
%
\end{Proposition}

\begin{proof}
For $\xi \in \fg^{\sharp}_{[\psi]}$, we have $\dd\rho(\xi)\psi = -i\mu(\xi) \psi$.
As
\[\dd\rho([\xi,\eta])\psi = [\dd\rho(\xi), \dd\rho(\eta)] \psi = 0\quad\text{for}\quad
\xi, \eta \in \fg^{\sharp}_{[\psi]}\,,\]
it follows that $-i\mu$ is an $i\R$-valued character.
Similarly, the smooth map
\[F\colon G^{\sharp} \rightarrow \C,\quad 
F(g) := \frac{\lra{\psi}{\rho(g)\psi}}{\lra{\psi}{\psi}}\] 
is a $\U(1)$-valued character when 
restricted to $G_{[\psi]}^{\sharp}$, as $\rho(g)\psi = F(g)\psi$
on that subgroup. 
In fact, $F \colon G^{\sharp} \rightarrow \C$ takes values in the unit ball
$\Delta \subseteq \C$, and $G_{[\psi]}^{\sharp}$ is the preimage 
of the unit circle $\partial \Delta$.
The derivative of $F$ at the unit $\one \in G$ is $D_{\one}F = -i\mu_{[\psi]}$, so
for any Lie subgroup $H \subseteq G^{\sharp}_{[\psi]}$, 
the restriction of $F$ to $H$ is a $\U(1)$-valued smooth character
that integrates $-i\mu_{[\psi]}|_{\mathrm{Lie}(H)}$.
\end{proof}
%

Note that the image of $\mu$ is contained in the hyperplane 
$(\fg^{\sharp})'_{-1}\subset (\fg^{\sharp})'$
of elements that evaluate to $-1$ on $1 \in 
\R = \mathrm{Ker}(\fg^{\sharp} \rightarrow \fg)$.
Now suppose that the image of $D_{[\psi]}\mu$ is dense in $T_{\mu_{[\psi]}} 
(\fg^{\sharp})'_{-1} = (\fg^{\sharp})'_{0} \simeq \fg'$. Since 
$\mathrm{Im} (D_{[\psi]}\mu) \subseteq (\fg^{\sharp}/ \fg^{\sharp}_{\psi})'$, we then have 
$\fg^{\sharp}_{\psi}= \{0\}$, so that $\fg^{\sharp}_{[\psi]} = \R$.
For points $[\psi] \in \bP(\cH^{\infty})$ where the image of $D_{[\psi]}\mu$ is dense, 
the identity component of any Lie subgroup 
$H \subseteq G_{[\psi]}$ is therefore $\U(1)$, and since the 
character on $\U(1)\subseteq G^{\sharp}_{[\psi]}$ is always the identity,
Proposition~\ref{IntegralCharacters} yields no extra information.

However, Proposition~\ref{IntegralCharacters} does yield nontrivial 
integrality requirements if $G^{\sharp}_{[\psi]}$ is strictly bigger than $\U(1)$,
which one expects to be the case 
for extremal points of the momentum set $\mathrm{Im} \mu$. Compare this to
Lemma~2.1 and Theorem~8.1 in \cite{GS82},
where it is shown that for compact Lie groups $G$, the vertices of the momentum polygon are integral lattice points in the dual $\mathfrak{h}'$ of the Cartan subalgebra.

\section{Diffeological Smoothness}\label{sec:diffeosmoothness}
As noted in the introduction, the action 
$G^{\sharp} \times \bP(\cH^{\infty}) \rightarrow \bP(\cH^{\infty})$ is separately smooth,
but not necessarily smooth.
However, if we settle for smoothness in the sense of diffeological spaces,
then one can hope for this action to be smooth
for the (large) class of
\emph{regular} Lie groups modelled on \emph{barrelled} spaces, which 
includes regular Fr\'echet and LF Lie groups.  
Here we prove the infinitesimal version of this, namely that 
the infinitesimal action
$\fg^{\sharp} \times \cH^{\infty} \rightarrow \cH^{\infty}$ is a smooth map of diffeological spaces.

\subsection{Infinitesimal action}
Let $\overline{\rho}$ be a smooth projective unitary representation of 
a locally convex Lie group $G$ modelled on a 
\emph{barrelled} Lie algebra $\fg$.
\begin{Lemma}\label{continuity}
If $\xi \colon \R^n \rightarrow \fg$ and $\psi \colon \R^{m} \rightarrow \cH^{\infty}$
are continuous, then the map $\dd\rho(\xi)\psi \colon \R^{n} \times \R^{m}\rightarrow \cH^{\infty}$ defined by 
$(s,t) \mapsto \dd\rho(\xi_{t})\psi_{s}$ is continuous.
\end{Lemma}
\begin{proof}
Since the Lie algebra action $\fg^{\sharp} \times \cH^{\infty} \rightarrow \cH^{\infty}$
is sequentially continuous by \cite[Lemma 3.14]{JN15}, the same holds 
for its concatenation with the continuous map
$(\xi,\psi) \colon \R^n \times \R^m \rightarrow \fg^{\sharp} \times \cH^{\infty}$.
Since $\R^n \times \R^m$ is first countable, this implies continuity.
\end{proof}

\begin{Lemma}\label{C1lemma}
If $\xi \colon \R^n \rightarrow \fg^{\sharp}$ and $\psi \colon \R^m \rightarrow \cH^{\infty}$ 
are $C^1$, then so is $\dd\rho(\xi)\psi$, and 
$D_{(v_1,v_2)}(\dd\rho(\xi)\psi)_{s,t} = \dd\rho(\partial_{v_1}\xi_{s})\psi_{t} + \dd\rho(\xi_{s})\partial_{v_2}\psi_{t}$.
\end{Lemma}
\begin{proof}
For the directional derivative along $(v_1,v_2) \in T_{s,t}(\R^{n}\times \R^{m})$,
note that 
\[
D_{v_1,v_2}(\dd\rho(\xi)\psi)_{s,t} =
\lim_{\varepsilon \rightarrow 0} \dd\rho(\Delta \xi_{s}(\varepsilon)) \psi_{t+\varepsilon v_2}
+ \dd\rho(\xi_{s})\Delta \psi_{t}(\varepsilon),
\]
with difference quotients $\Delta\xi$ and $\Delta \psi$ defined by  
$\Delta \xi_{s}(\varepsilon) := \frac{1}{\varepsilon}(\xi_{s+\varepsilon v_1} - \xi_{s})$
and
$\Delta \psi_{t}(\varepsilon) := \frac{1}{\varepsilon}(\psi_{t+\varepsilon v_2} - \psi_{t})$
for $\varepsilon \neq 0$, and 
$\Delta \xi_{s}(0) := \partial_{v_1}\xi_{s}$ 
and
$\Delta \psi_{t}(0) := \partial_{v_2}\psi_{t}$ for $\varepsilon =0$. 
Since $\Delta\xi$ and $\Delta\psi$ are continuous in $\varepsilon$, the formula 
for $D_{v_1,v_2}(\dd\rho(\xi)\psi)_{s,t}$ follows by Lemma \ref{continuity}.
Another application of this lemma to $(D\xi,\psi)$ and $(\xi,D\psi)$ shows that 
the derivative is continuous.
\end{proof}
 
\begin{Proposition}\label{PropCinfty}
If $\xi \colon \R^n \rightarrow \fg^{\sharp}$ and $\psi \colon \R^m \rightarrow \cH^{\infty}$ 
are $C^k$ for $k\in \N$ or $k=\infty$, then so is $\dd\rho(\xi)\psi$.
\end{Proposition}
\begin{proof}
This follows by induction on $k$, 
using Lemmas \ref{continuity} and~\ref{C1lemma}.
\end{proof}

If we equip all locally convex manifolds $\mathcal{M}$ with the diffeology of smooth 
maps from open subsets of Euclidean space into $\mathcal{M}$, then 
the following is simply a reformulation of Proposition \ref{PropCinfty}.
\begin{Proposition}
If $G$ is modelled on a barrelled Lie algebra $\fg$, then
the infinitesimal action
$\fg^{\sharp} \times \cH^{\infty} \rightarrow \cH^{\infty}$
is a smooth map of diffeological spaces.
\end{Proposition}

\bibliographystyle{alpha}

\begin{thebibliography}{MW04} 

\bibitem[AL92]{AL92} Arnal, D., and J. Ludwig, {\it La convexit\'e de
l'application moment 
d'un groupe de Lie}, J. Funct. Anal. {\bf 105:2} (1992), 256--300

\bibitem[GS82]{GS82}
Guillemin, V., and S.\ Sternberg,
{\it Convexity properties of the moment mapping},
Invent.\ Math. {\bf 67} (1982), 491--513


\bibitem[JN15]{JN15}
Janssens, B.\ and K.-H. Neeb,
{\it Projective unitary representations of infinite dimensional Lie groups},
{\tt arXiv:1501.00939}, 2015

\bibitem[Mi80]{Mi80}
Michor, P.\ W., ``Manifolds of Differentiable Mappings'',
Shiva Publishing Ltd., Kent, 1980

\bibitem[Mi90]{Mi90}
---,  {\it The moment mapping for a unitary representation},
Ann. Global Anal. Geom. {\bf 8:3} (1990), 299--313


\bibitem[MW04]{MW04}
Mohrdieck, S.\ and R.\ Wendt,
{\it Integral conjugacy classes of compact Lie groups},
Manuscripta Math.\ {\bf 113:4} (2004), 531--547

\bibitem[Ne95]{Ne95} Neeb, K.-H., 
{\it On the convexity of the moment mapping for a unitary highest 
weight representation}, J. Funct. Anal. {\bf 127:2} (1995), 301--325 

\bibitem[Ne00]{Ne00} ---, ``Holomorphy and Convexity in Lie Theory,'' 
Expositions in Mathematics {\bf 28}, de Gruyter Verlag, Berlin, 1999 

\bibitem[Ne01]{Ne01} ---, {\it Representations of infinite dimensional
  groups}, pp.~131--178; in ``Infinite Dimensional K\"ahler Manifolds,'' 
Eds. A. Huckleberry, T. Wurzbacher, DMV-Seminar {\bf 31}, 
Birkh\"auser Verlag, 2001 

\bibitem[Ne10a]{Ne10a} ---, {\it Semibounded representations and invariant 
cones in infinite dimensional Lie algebras}, Confluentes Math. {\bf 2:1} 
(2010), 37--134

\bibitem[Ne10b]{Ne10b} ---, 
{\it On differentiable vectors for representations of infinite dimensional Lie 
groups}, J. Funct. Anal. {\bf 259} (2010), 2814--2855

\bibitem[NV03]{NV03} Neeb, K.-H., and C. Vizman, {\it Flux homomorphisms and 
principal bundles over infinite-dimensional manifolds}, 
Monatshefte f\"ur Math. {\bf 139} (2003), 309--333 


\bibitem[Wi92]{Wi92} Wildberger, N., 
{\it The moment map of a Lie group representation}, 
Trans.\ Amer. Math.\ Soc.\ {\bf 330} (1992), 257--268 

\end{thebibliography}

\end{document}